\documentclass[a4paper,abstracton]{scrartcl}
\usepackage[utf8]{inputenc}
\usepackage{amsfonts,amsthm,amssymb,amsmath}
\usepackage[inline]{enumitem}
\usepackage{color}
\usepackage{graphicx} 
\usepackage{authblk}

\newtheorem{thm}{Theorem}
\newtheorem{lem}[thm]{Lemma}
\newtheorem{cor}[thm]{Corollary}
\newtheorem{conj}[thm]{Conjecture}

\theoremstyle{remark}

\title{List majority edge-colorings of graphs}
\author{Rafał Kalinowski, Monika Pilśniak\footnote{Corresponding author: pilsniak@agh.edu.pl}, Marcin Stawiski
} 

\affil{AGH University of Krakow, \protect\\Department of Discrete Mathematics, \protect\\ al. Mickiewicza 30, 30-059 Krakow, Poland}

\date{}
\begin{document}
\maketitle
\begin{abstract}
A majority edge-coloring of a graph without pendant edges is a coloring of its edges such that, for every vertex $v$ and every color $\alpha$, there are at most as many edges incident to $v$ colored with $\alpha$ as with all other colors. We extend some known results for finite graphs to infinite graphs, mostly in the list setting. In particular, we prove that every infinite graph without pendant edges has a majority edge-coloring from lists of size $4$. Another interesting result states that every infinite graph without vertices of finite odd degrees admits a majority edge-coloring from lists of size $2$. We formulate two conjectures. As a consequence of our results, we prove that line graphs of any cardinality admit majority vertex-colorings from lists of size 2, thus confirming the Unfriendly Partition Conjecture for line graphs.
\end{abstract}

\section{Introduction}

For a graph $G$, an edge-coloring $c:E(G)\to [k]$ is a {\it majority $k$-edge-coloring} if, for every vertex $u$ of $G$ and every color $\alpha$ in $[k]$, the cardinality of edges incident to $v$ colored with $\alpha$ is not greater than the cardinality of edges incident to $v$ colored with all other colors. 
Of course, graphs with pendant edges do not admit such a coloring. As usually, the least number of colors in a majority edge-coloring is of interest. The concept of majority edge-colorings was recently introduced in \cite{BKPPRW}, and was motivated by {\it majority vertex-colorings}, i.e. colorings of vertices of a graph $G$ such that each vertex $v$ has at most as many neighbors with the color of $v$ as with other colors. They were studied already in 1966 by Lov\'asz in \cite{Lov-66} for finite graphs, then extended to infinite graphs (e.g. in \cite{Berger,BDGS, SheMil}), and more recently to digraphs in \cite{KOSvW}. This problem for infinite graphs, including the well-known Unfriendly Partition Conjecture, is presented in more details in Section~\ref{UPC}. 
\vspace{4mm}

In~\cite{BKPPRW}, the following results were proved for finite graphs.
\begin{thm}\label{general}{\rm (\cite{BKPPRW})}
Every finite graph of minimum degree at least $2$ admits a majority $4$-edge-coloring.
\end{thm}

\begin{thm}\label{fin3}{\rm (\cite{BKPPRW})}
If $G$ is a finite graph of minimum degree $\delta(G)\geq 4$, then $G$ admits a majority $3$-edge-coloring.
\end{thm}

\begin{thm}\label{mi2}{\rm (\cite{BKPPRW})}
A finite graph $G$ admits a majority $2$-edge-coloring if and only if all vertices of $G$ have even degrees and the size of $G$ is even.
\end{thm}

In this paper, we extend these results to infinite graphs whose orders are arbitrary cardinal numbers, considering also the list version of edge-colorings.  Section~\ref{auxr} contains some auxiliary results for Section~\ref{3}, where we prove the list version of Theorem~\ref{mi2} for finite and infinite graphs. Next in Section~\ref{UPC}, we confirm the well-known Unfriendly Partition Conjecture for line graphs, what we derive from the results of the previous section. Actually, we prove there a stronger result that line graphs of any cardinality admit majority vertex-colorings from lists of size 2. In Section~\ref{inf}, we show that Theorem~\ref{general} and Theorem~\ref{fin3} hold for infinite graphs, the first one also in the list setting.  

\hspace{5mm}

We use standard terminology and notation of graph theory (cp. \cite{Diestel}). A {\it double ray} in an infinite graph is a two-sided infinite path, i.e. an infinite connected 2-regular graph. A {\it ray} is a one-sided infinite path, and a unique vertex of degree one in a ray is its {\it endvertex}. 

Following Schmidt~\cite{Schmidt}, we define the {\it rank} of a graph without rays as follows. We assign rank 0 to all the finite graphs. Next, we assign the smallest ordinal $\rho$ as the rank of any infinite graph $G$ which does not already have rank less than $\rho$ and which contains a finite set $S$ of vertices such that each component of $G-S$ has some rank less than $\rho$. We shall make use of the following facts (cf.~\cite{Schmidt, Halin, BDGS}). For every rayless graph $G$ with rank $\rho(G)>0$, there exists a minimal set $S$ that works, called a {\it core} of $G$. 
If $S$ is the core of $G$, then the number of components of $G-S$ is infinite, and each vertex of $S$ has infinite degree in $G$. Moreover, if $C_1,\ldots,C_n$ is a finite set of components of $G-S$, then the rank of the subgraph of $G$ induced by $S\cup\bigcup_{i=1}^n V(C_i)$ has rank less than the rank of $G$.

For convenience, we introduce here some more notation and terminology. If $W$ is an oriented walk of a graph and $e$ is an edge in $W$, then we denote by $e^+$ the subsequent edge, and by $e^-$ the previous edge in $W$. 

Assume that an edge-coloring of a graph is given. 
We say that a vertex $v$ is \emph{majority colored} if there is no color $\alpha$ such that $v$ is incident with more edges of color $\alpha$ than with the remaining incident edges.  Note that if the degree of $v$ is infinite, then $v$ is majority colored if and only if for every color $\alpha$, the cardinality of the set of edges incident with $v$ which are not colored with $\alpha$ equals the degree of $v$. We say that a vertex $v$ is \emph{almost majority colored} if $v$ is not majority colored and there exists a color $\alpha$, called an {\it overwhelming color}, such that $v$ is incident with at most two edges more of the color $\alpha$ than the remaining incident edges. Notice that if $v$ is an almost majority colored vertex, then the degree of $v$ is finite and the number of its incident edges with the overwhelming color exceeds the number of remaining incident edges by 1 or 2, depending on the parity of the degree of $v$. If all vertices of a graph $G$ are majority colored, then  the graph $G$ is majority colored. 


\section{Auxiliary results}\label{auxr}

\begin{lem}\label{Lem:xd}
    Let $G$ be a finite graph of even size with all vertices of even degree, and $\mathcal{L}$ be a set of lists for edges of $G$, each of size $2$.  Then there exists a majority edge-coloring  of $G$ from the set $\mathcal{L}$ of lists.
\end{lem}
\begin{proof}
It follows from the assumptions that $G$ has an Euler tour $W$. We fix an orientation of $W$. First, assume that all elements of $\mathcal{L}$ are the same. We color the edges in $W$ alternately. Clearly, this yields a majority coloring.

If not every list in $\mathcal{L}$ is the same, then we choose two consecutive edges $e$ and $e^+$ in $W$ which have different lists. We color $e^+$ with a color which does not belong to $L(e)$. Next, we color each consecutive edge in $W$ starting from the edge subsequent to $e^+$ in such a way that, for every edge $f$, the next edge $f^+$ has a different color than $f$ has. The choice of $e$  guaranties that each two consecutive edges in $W$ have distinct colors. Therefore, the obtain coloring is a majority coloring.
\end{proof}

\begin{lem}\label{Lem:XD}
    Let $G$ be a finite graph of odd size with all vertices of even degree, and $\mathcal{L}$ be a set of lists for edges of $G$, each of size $2$. Let $b$ be an arbitrary vertex of $G$. Then there exists an edge-coloring  of $G$ from the set $\mathcal{L}$ of lists such that each vertex different from $b$ is majority colored, and either $b$ is majority colored, or $b$ is almost majority colored.
\end{lem}
\begin{proof}
Again, $G$ has an Euler tour $W$.  We color the edges in $W$ starting from an edge incident to $b$ in such a way that if $e$ has a color $\alpha$, then $e^+$ has a different color, except possibly the first and the last edge of $W$, which are incident to $b$. Now, each vertex, possibly except $b$, is majority colored, and $b$ may have exactly two incident edges of one color more than the edges of all other colors.
\end{proof}
In the proof of the main result, we make use of the following theorem, which may be of interest on its own.

\begin{thm}\label{Lem:1}
Let $G$ be a graph and let  $\mathcal{H}$ be any family of graphs. Then $G$ contains a subgraph $F$  decomposable into elements of $\mathcal{H}$, which is maximal among all such subgraphs, that is, the remaining subgraph $G-E(F)$ contains no subgraph isomorphic to a graph from $\mathcal{H}$.
\end{thm}
\begin{proof}
 The idea of the proof is to show that there is a maximal decomposition which gives us a subgraph with the desired property rather than straightforwardly proving that such a maximal subgraph exists. First, we define a family $\mathcal{F}$ of subsets of $E(G)$ which are decompositions of subgraphs of $G$ into unions of some pairwise edge-disjoint elements of $\mathcal{H}$. For $A,B\in \mathcal{F}$, we write $A\leq B$ if $A\subseteq B$. We show that every non-empty chain in $(\mathcal{F}, \leq)$ contains an upper bound. Let $\mathcal{C}$ be a non-empty chain in $\mathcal{F}$, and let $M=\bigcup \mathcal{C}$. For every element $C\in \mathcal{C}$, we have $C\leq M$, and $M$ is a decomposition of a certain subgraph of $G$. Hence $M$ majorizes $C$. Therefore, by Zorn's Lemma there exists a maximal element in $\mathcal{F}$. The obtained maximal decomposition gives a maximal subgraph from the statement of the theorem. 
\end{proof}


\section{Lists of size 2}\label{3}

Let us formulate the main result of this section.

\begin{thm}\label{2-list}
A graph $G$ of arbitrary order admits a majority edge-coloring from any set $\mathcal{L}$ of lists, each of size $2$, if and only if no vertex of $G$ has odd degree and $G$ is not a~finite graph of odd size.
\end{thm}

The "only if" part is obvious. For the "if" part, we prove the following a bit stronger result.

\begin{thm}\label{main}
    Let $G$ be a graph, and $\mathcal{L}$ be the set of lists for the edges of $G$, each of size $2$. Then there exists an edge-coloring of $G$ from the set $\mathcal{L}$ of lists  in which every vertex of $G$ which is not of  odd degree is majority colored and each vertex of odd degree in $G$ is either majority colored or almost majority colored, or $G$ is a finite graph of odd size with all vertices of even degree.
\end{thm}
\begin{proof}
Let us call an edge coloring from the lists $\mathcal{L}$ {\it good} if it satisfies the claim, i.e. only vertices of odd degrees need not be majority colored, but they have to be almost majority colored.

The case when $G$ is finite follows from Lemma \ref{Lem:xd} and Lemma \ref{Lem:XD}. Assume  that $G$ is infinite.

If $G$ contains a double ray, then let $F$ be a maximal subgraph of $G$ decomposable into double rays, which exists by Lemma \ref{Lem:1}. Let us fix a decomposition $\mathcal{P}$ of $F$ into double rays. We color each double ray $P\in\mathcal{P}$  in such a way that if an edge $e$ of $P$ has a color $\alpha$, then each of two its adjacent edges, $e^-,e^+$ in $P$, has a color different than $\alpha$.  Hence every vertex of $F$ is majority colored in $F$ from the set $\mathcal{L}$ of lists. 

Now, we consider every component $H$ of $G-E(F)$. Clearly, each vertex $v\in V(F)$ belongs to exactly one such component (which can be $K_1$). Observe also that no vertex has odd degree in $F$. Consequently, if every component $H$ admits a good coloring from the set $\mathcal{L}$ of lists, then the whole graph $G$ also does. The only components that need a special treatment are finite components of odd size with all vertices of even degree, because they contain vertices which may be almost majority colored (cp.  Lemma~\ref{Lem:XD}).
 
Let $X$ be the set of vertices with odd degrees in $H$. To color the edges of $H$, we distinguish three cases.

\textbf{Case 1.} Component $H$ is a finite graph.

\noindent If the set $X$ is empty and $H$ has even size, then, by Lemma~\ref{Lem:xd}, there exists a majority edge-coloring of $H$ from the set $\mathcal{L}$ of lists. Clearly, each vertex of $H$ is now majority colored also in $G$.

Now let the set $X$ be non-empty. Notice that  $X$ has an even number of vertices because $H$ is finite. Hence, we can partition $X$ into pairs. For each such pair $\{ x,y \}$, we join $x$ and $y$ by an extra path of length at most three to obtain a graph $H'$ of even size and with all vertices of even degree. Then we color $H'$ with a majority edge-coloring from the list set $\mathcal{L}$, which exists by Lemma \ref{Lem:xd}. Deletion of the additional paths yields a good edge-coloring of the subgraph $G[V(F)\cup V(H)]$ since the only vertices of $H$ with deleted incident edges belong to $X$, and each such vertex is either majority colored or almost majority colored in this subgraph. 

Assume now that the set $X$ is empty but the size of $H$ is odd. Let $b$ be any vertex of $H$ which also belongs to a double ray $P\in \mathcal{F}$.  It follows from Lemma~\ref{Lem:XD} that all vertices of $H$, except possibly $b$, are majority colored. We now show how to finally make $b$ majority colored.
\vspace{3mm}

Suppose that all finite components of $G-E(F)$ are already colored in the manner just described. Let $Y$ denote the set of vertices of the subgraph $G'$ induced by the vertices of $F$ and all finite components of $G-E(F)$, which are not majority colored and have finite even degree. Thus each vertex $b\in Y$ is a common vertex of a double ray $P\in \mathcal{F}$ and a finite component $H$ of odd size with all vertices of even degrees in $H$. Moreover, there is a color $\alpha$ such that the number of edges incident to $b$ colored with $\alpha$ exceeds the number of incident edges with other colors exactly by 2. To make all vertices of $Y$ majority colored, for each $P\in \mathcal{F}$ we perform the following procedure of {\it shifting the edge-coloring} of $P$. We enumerate the vertices of $P$ by integers, i.e. $V(P)=\{\ldots, v_{-1},v_0,v_1,\ldots\}$. For $i=0,1,2,\ldots,$ we verify whether $v_i\in Y$. If this is the case, then $v_i$ is overwhelmed by a color of an edge $e$ of $P$ incident to $v_i$. We can assume that $e=v_iv_{i+1}$. Indeed, if $e=v_{i-1}v_i$, then we recolor the edges of $H$ such that $v_i$  is either majority colored or overwhelmed by the color of the edge $v_iv_{i+1}$. This can be done because each list of $\mathcal{L}$ contains two colors. Next we put another admissible color from $L(e)$ on $e=v_iv_{i+1}$, and we consecutively recolor the edges $v_jv_{j+1}$ of $P$, for $j=i+1,i+2,\ldots,$ such that the color of $v_jv_{j+1}$ is distinct from the color of the previous edge $v_{j-1}v_j$. Next,  we examine the vertices $v_{-i}$ for $i=1,2,\ldots$. If $v_{-i}\in Y$, then we may analogously assume that $v_i$ is overwhelmed by the color of the edge $v_{-i-1}v_{-i}$. Then we consecutively recolor the edges $v_{-j-1}v_{-j}$, for $j=i,i+1,\ldots,$ such that the color of $v_{-j-1}v_{-j}$ is distinct from that of the edge $v_{-j}v_{-j+1}$.

Thus, we obtain a good coloring of the subgraph $G'$, the sum of $F$ and all finite components of $G-E(F)$.

\vspace{5mm}

\textbf{Case 2.} Component $H$ is an infinite rayless graph.

\noindent We proceed by transfinite induction on the rank of $H$.  If the rank of $H$ is $0$, then $H$ is a finite graph, which we consider in Case 1. Assume then that the rank of $H$ is greater than 0. Let $S$ be the core of $H$. Then each vertex of $S$ has infinite degree in $H$. Let $\mathcal{C}$ be the set of all components of $H-S$. For $C\in\mathcal{C}$, denote $C'=H[S\cup V(C)]$, the subgraph induced by the set $S$ and the vertices of the component $C$. 

Denote by $S_1$ the set of all vertices $s\in S$ which are connected with $d_H(s)$ components from $\mathcal{C}$ by exactly one edge. For each $s\in S_1$, select a set $\mathcal{C}(s)$ of cardinality $d_H(s)$ of such components. As the set $S_1$ is finite, this can be done in such a way that the sets $\mathcal{C}(s), s\in S_1,$ are pairwise disjoint. Next, partition each set $\mathcal{C}(s)$ into pairs $(C_1,C_2)$. For each such pair take a good edge-coloring of the subgraph $H[S\cup V(C_1)\cup V(C_2)]$, in which $s$ has degree 2. Such a good coloring exists by the induction hypothesis. It is easy to see that each vertex of $S_1$ is already majority colored no matter how the remaining edges incident to it will be colored.   

Now, consider the components from $\mathcal{C}\setminus \bigcup_{s\in S_1}\mathcal{C}(s)$. For each such component $C$ take a good coloring of $C'$ which exists by the induction hypothesis. Thus all edges in $H$ are colored.

Observe that each vertex $v$ of any component $C\in \mathcal{C}$ is majority colored, unless $v$ is of odd degree in $H$ and $v$ is almost majority colored. This is because $N_H(v)\subset V(C')$. Also, every vertex $s\in S\setminus S_1$ is majority colored, even if its degree is odd for infinitely many components $C\in \mathcal{C}$. 

\vspace{5mm}

\textbf{Case 3.} Component $H$ is an infinite graph with a ray. 

Let $F'$ be a maximal subgraph of $H$ decomposable into rays whose existence follows from Lemma~\ref{Lem:1}. We select a decomposition $\mathcal{R}$ of $F'$ into rays. Observe that the decomposition $\mathcal{R}$ may contain any number of rays, but any two rays share infinitely many vertices since $F$ does not contain a double ray. 

This time, we first color components of $H-E(F')$. Each such component $K$ is a finite graph or an infinite rayless graph. We color its edges  in the same way as in Case 1 or Case 2, respectively. Thus every vertex $v\in V(K)$ is majority colored in $K$ unless either $v$ has odd degree in $K$ or $K$ is a finite graph of odd size with all vertices of even degree, where $v$ is an almost majority colored common vertex of $K$ and a ray from $\mathcal{R}$. 

Now we color the edges of $F'$. Clearly, a vertex $v$ has an odd degree in $F'$ only if $v$ is an endvertex of a ray from $\mathcal{R}$. For each such vertex $v$, let $r_v$ be the cardinality of the set $\mathcal{R}(v)$ of rays in $\mathcal{R}$ starting at $v$. Suppose first that $v$ has even or infinite degree in its component $K$ of $H-E(F')$. If $r_v$ is an infinite cardinal or an even natural number, then we partition the set $\mathcal{R}(v)$ into two subsets $\mathcal{R}^1(v)$ and $\mathcal{R}^2(v)$ of the same cardinality $r_v$. Thus, there is a one-to-one correspondence between elements of these two sets. We choose colors for initial edges of the rays from $\mathcal{R}(v)$ in such a way that if the initial edge of a ray $R_1\in \mathcal{R}^1(v)$ has color $\alpha$, then the initial edge in the corresponding ray $R_2\in \mathcal{R}^2(v)$ gets a different color. Thus $v$ is majority colored in the whole graph $G$.  When $r_v$ is an odd natural number, then we similarly partition the set $\mathcal{R}(v)$ into two subsets of sizes differing by one, and choose colors for initial edges of rays of $\mathcal{R}(v)$ such that either $v$ is majority colored or almost majority colored in the component $H$, and hence in the whole graph $G$. 

Let $v$ have an odd degree in its component $K$ of $H-E(F')$. We analogously partition the set $\mathcal{R}(v)$ into subsets $\mathcal{R}^1(v)$ and $\mathcal{R}^2(v)$, where the size of $\mathcal{R}^1(v)$ can be greater by one than the size of $\mathcal{R}^2(v)$ if $r_v$ is odd. If $v$ is almost majority colored in $K$ with the overwhelming color $\alpha$, then we assume that initial edges of rays in $\mathcal{R}^1(v)$ cannot get color $\alpha$. Then we proceed as above.

Thus all endvertices of rays in $\mathcal{R}$ are well colored in $F'$. Next, we color every ray $R$ in such a way that each edge $e$ gets a color distinct from that of the previous edge $e^-$ in $R$. 

However, not all vertices of $H$ are well colored. This concerns the set $Y$ of common vertices of rays and finite components of $H-E(F')$ of odd size with all vertices of even degree. To correct this flaw, we perform the procedure of shifting colorings of rays, in fact the same as for double rays in the first stage of the proof. Namely, for every ray $R\in \mathcal{R}$, we enumerate its consecutive vertices with non-negative integers, i.e. $V(R)=\{v_i: i=0,1,\ldots\}$. We find the smallest $i$ such that $v_i\in Y$.  Thus $v_i$ is overwhelmed in $H$ by a certain color $\alpha$. Without loss of generality, we may again assume that $\alpha$ is a color of the edge $v_iv_{i+1}$ in $R$. We change the color of $v_iv_{i+1}$, thus making $v_i$ majority colored. Then we recolor consecutive edges $v_jv_{j+1}, j=i+1,i+2,\ldots$, of the ray $R$ with a color distinct from the color of $v_{j-1}v_j$. Next we find another smallest $i$ with $v_i\in Y$, and continue the procedure of shifting colorings of $R$ until all vertices of $R$ are examined. 

This yields a good edge-coloring of $H$, and thus of the whole graph $G$.

\end{proof}

We conclude with an obvious consequence of Theorem~\ref{2-list} concerning usual majority colorings without lists.

\begin{cor}\label{maj2}
A graph $G$ of arbitrary order admits a majority $2$-edge coloring if and only if no vertex of $G$ has odd degree and $G$ is not a finite graph of odd size.
\end{cor}

\section{Unfriendly Partition Conjecture holds for line graphs}\label{UPC}

Recall that a majority $k$-vertex-coloring of a graph $G$ is a mapping $c:V(G)\to [k]$ such that, for every vertex $v\in V(G)$, the cardinality of neighbors of $v$ with the color $c(v)$ is not greater than the cardinality of neighbors in other colors. Equivalently, a graph $G$ admits a majority $k$-vertex-coloring if there is a partition $V_1,\ldots,V_k$ of $V(G)$ such that, for each vertex $v\in V_i$, one has $|N(v)\cap V_i|\leq |N(v)\cap(V\setminus V_i)|$, for $i=1,\ldots,k$. 

Cowan and Emerson, in unpublished work, conjectured that every infinite graph has a majority 2-vertex coloring. It was disproved by Shelah and Milner by showing the following. 
\begin{thm} {\rm (\cite{SheMil})}
There exist uncountable graphs without majority $2$-vertex-coloring. However, every infinite graph has a majority $3$-vertex-coloring.
\end{thm}

The question whether countably infinite graphs have a majority 2-vertex-coloring remains open, and is known as the Unfriendly Partition Conjecture, which reads as follows.

\begin{conj} Every countably infinite graph admits a majority $2$-vertex-coloring.
\end{conj}

The Unfriendly Partition Conjecture has been confirmed for graphs \\ $\bullet$ with finitely many vertices of infinite degree~(\cite{AhMiPr}),
\\ $\bullet$ with all vertices of infinite degree~(\cite{DeV}),
\\ $\bullet$ without rays~(\cite{BDGS}),
\\ $\bullet$ without a subdivision of infinite clique~(\cite{Berger}). 
\vspace{3mm}

Let us add that in 2023 Haslegrave proved the following.
\begin{thm} {\rm (\cite{Has})}
Every countable graph admits a majority vertex-coloring from any lists of size $3$.     
\end{thm} 

Our Theorem \ref{main} easily implies the following result.
\begin{thm}
Every line graph of arbitrary infinite order admits a majority vertex-coloring from any lists of size $2$.
\end{thm}
\begin{proof}
Let $G$ be a line graph of a graph $H$. Let $\mathcal{L}$ be any set of lists for vertices of $G$, each of size 2. Simultaneously, $\mathcal{L}$ is the set of lists for edges of $H$. By Theorem~\ref{main}, there exists an edge-coloring $c$ of $H$ from the set $\mathcal{L}$ of lists such that all vertices are majority colored, except possibly, some vertices of odd degree which are almost majority colored. Naturally, we color each vertex $e$ of $G$ with the color $c(e)$ of the corresponding edge $e=uv$ of $H$. It is easy to see that the cardinality of neighbors of $e$ in $G$ colored with $c(e)$ cannot be greater than that with other colors, even if one or two endvertices $x,y$ of $e$ have odd degree in $H$ and $c(e)$ is an overwhelming color. Thus we obtained a majority 2-vertex-coloring of the graph $G$. \end{proof}

Consequently, countable line graphs form another class of graphs satisfying the Unfriendly Infinite Conjecture.
\begin{cor}
The Unfriendly Infinite Conjecture holds for line graphs.     
\end{cor}


\section{General bound}\label{inf}

In this section, we discuss a general upper bound for the least size of lists  in a majority edge-coloring of graphs of arbitrary order. Note that this problem is related to the well-known List Coloring Conjecture which is still open. However, for our purpose the following result of Borodin, Kostochka and Woodall \cite{BKW} suffices.

\begin{thm}{\rm (\cite{BKW})}
Let $G$ be a finite graph with maximum degree $\Delta(G)$, and let $\mathcal{L}$ be a set of lists, each of size $\lfloor\frac32\Delta(G)\rfloor$ assigned to edges of $G$. Then $G$ has a proper edge-coloring from these lists. 
\end{thm}

Using Tychonoff's theorem in a standard way, it is easy to show that this theorem is also true for infinite graphs. This result enables us to prove the following lower bound for the size of lists allowing a majority edge-coloring.

\begin{thm}
Let $G$ be a graph of arbitrary order and without pendant edges. Then $G$ admits a majority edge-coloring from any collection of lists, each of size $4$.
\end{thm}

\begin{proof}
Given a graph $G$, analogously as in \cite{BKPPRW} we construct a graph $G^*$ in the following way.
We split every vertex $v$ of degree greater than 3 in $G$ into a set of vertices $\{v_i:i\in I(v)\}$ of degrees 2 or 3 in $G^*$ by a suitable partition of its neighborhood $N_G(v)$. Naturally, if the degree $d_G(v)$ is infinite, then the cardinality of the index set $I(v)$ equals $d_G(v)$. Observe that there is a one-to-one correspondence between the edges of $G$ and $G^*$. So the list for an edge of $G^*$ is the same as the list for its counterpart in $G$. Each component of the graph $G^*$ is a countable subcubic graph. By the above result of Borodin, Kostochka and Woodall for countable graphs, there exists a proper coloring  of $G^*$ from any collection of lists of size 4. This coloring  transferred to the graph $G$ yields a majority coloring since for every $v\in V(G)$, the edges incident to any single vertex of $\{v_i:i\in I(v)\}$ have distinct colors. 
\end{proof}

This bound is tight. Namely, for each infinite cardinal $\kappa$, there exists a graph of order $\kappa$ which needs four colors in any majority edge-coloring. To see this, take any graph $G'$ of order $\kappa$, and a cubic finite graph $H$ of Class 2 which contains a vertex $v\in V(H)$ such that for every its incident edge $vw$ the subgraph $H-vw$ is still of Class 2 (Petersen graph is an example). Now, construct a graph $G$ by joining the graphs $G'$ and $H$ by an edge $uv$ where $u$ is any vertex of $G'$. It is easily seen that the graph $G$ does not admit a majority 3-edge-coloring since it would give a proper 3-edge-coloring of the subgraph $H-vw$, where $vw$ is an edge incident to $v$ in $H$.
We conjecture that subcubic Class 2 subgraphs are the only obstacle for the existence of a majority 3-edge-coloring, also for infinite graphs. For finite graphs, this conjecture was proposed in~\cite{BKPPRW}.

\begin{conj}
Every graph $G$ without pendant edges admits a majority $3$-edge-coloring unless $G$ contains a Class $2$ graph $H$ with $\Delta(H)=3$ as an induced subgraph, and with some vertices of $H$ of degrees at most $3$ in $G$.
\end{conj}

The following extension of Theorem~\ref{fin3} for infinite graphs supports this conjecture.

\begin{thm}\label{inf3}
If $G$ is an infinite graph of minimum degree $\delta(G)\geq 4$, then $G$ admits a majority $3$-edge-coloring.
\end{thm}

In the proof, we will make use of the following two lemmas.

\begin{lem}\label{f3}
If $H$ is a finite graph, then there exists a $3$-edge coloring of $H$ for which every vertex $v$ of degree $d_H(v)\geq 4$ is majority colored.
\end{lem}

\begin{proof}
To every vertex $u$ with $d_H(u)<4$ we attach an additional graph $K_5$, such that $u$ is the unique common vertex of $H$ and the $K_5$.  Thus we obtain a supergraph $H'$ with minimum degree at least 4. By Theorem~\ref{fin3}, there exists a majority 3-edge-coloring of $H'$. After removing all additional graphs $K_5$ we get a desired 3-edge coloring of $H$.
\end{proof}

\begin{lem}\label{count3}
Let $H$ be a countably infinite graph that contains a connected subgraph $F$ of vertices with finite degrees at least $4$ in $H$ such that each vertex $v\in V(H)\setminus V(F)$ has a neighbor in $F$. Then $H$ admits a $3$-edge-coloring such that all vertices in $F$ and all vertices of infinite degree in $H$ are majority colored. In particular, every  locally finite graph $H$ with $\delta(H)\geq 4$ has a majority $3$-edge-coloring.
\end{lem}

\begin{proof}
We pick a vertex $v_0$ in $F$. For every positive integer $n$, we consider the ball $B_n$ in $F$ of radius $n$ centered at $v_0$, i.e. the subgraph of $F$ induced by all vertices $v\in V(F)$ such that there is a $v_0v$-path of length at most $n$ in $F$. Furthermore, let $B'_n=H[V(B_n)\cup N_H(B_n)]$. Thus $B'_n$  contains all vertices of the ball $B_{n+1}$ in $F$, as well as every vertex outside $F$ having a neighbor in $B_n$. Observe that $d_{B'_n}(v)=d_H(v)\geq 4$ for each vertex $v\in V(B_n)$. By Lemma~\ref{f3}, there exists a 3-edge-coloring of $B'_n$ such that each vertex with degree at least 4 is majority colored. 

Let $V_n$ denote the set of all 3-edge-colorings of $B'_n$ such that every vertex of degree at least 4 in $B'_n$ is majority colored. We have just shown that $V_n\neq \emptyset$. Moreover, if $c\in V_n$, then its restriction to $B'_{n-1}$ belongs to $V_{n-1}$. We now define a graph $\Gamma$ with a vertex set $V(\Gamma)=\bigcup_{n=1}^{\infty}V_n$ by inserting all edges $cc'$ between $c\in V_n$ and $c'\in V_{n-1}$ whenever $c'$ is the restriction of $c$ to $B_{n-1}$. Hence, every vertex in $V_n$ has a neighbor in $V_{n-1}$. By K\H{o}nig's Infinity Lemma, the graph $\Gamma$ contains a ray $c_1c_2\ldots$, where $c_n\in V_n$. Then $\bigcup_{n=1}^{\infty} c_n$ is a 3-edge-coloring of $H$ such that every vertex $v$ of degree $d_H(v)\geq 4$ is majority colored. To see that every vertex $v$ of infinite degree in $H$ is majority colored, it suffices to note that $d_{B'_n}(v)\geq 4$ for almost all $n$. 
\end{proof}

Now, we can proceed to the proof of Theorem~\ref{inf3}.

\begin{proof} 
If $G$ is locally finite, then the claim follows from Lemma~\ref{count3}. Otherwise, let $V_{\infty}$ be the set of all vertices with infinite degree in $G$, and let $\mathcal{B}$ be the family of all components of the subgraph of $G$ induced by vertices of finite degrees. Clearly, each such component is countable. Moreover, let 
$$V_{\infty}^1=\{v\in V_{\infty}: d_G(v)=d_{G[V_{\infty}]}(v)\}.$$
By Corollary~\ref{maj2}, the subgraph $G[V_{\infty}]$ admits a 2-edge-coloring such that every vertex of infinite degree is majority colored. In particular, every vertex of $V_{\infty}^1$ is already majority colored regardless of colors of the other edges incident to it.  

Now, for every vertex $v\in V_{\infty}\setminus V_{\infty}^1$, consider the subfamily $\mathcal{B}(v)$ of all components $B\in \mathcal{B}$ such that the degree of $v$ in $G[V(B)\cup\{v\}]$ is positive but smaller than 4. Let $V_{\infty}^2$ denote the set of vertices $v$ in $V_{\infty}\setminus V_{\infty}^1$ such that the cardinality of the set of neighbors of $v$ in $\bigcup (\mathcal{B}\setminus \mathcal{B}(v))$ is smaller than $d_G(v)$. To each vertex $v\in V_{\infty}^2$ we assign a set $\mathcal{B}_0(v)$ of $d_G(v)$ components from $\mathcal{B}(v)$ in such a way that distinct vertices from $V_{\infty}^2$ get disjoint sets $\mathcal{B}_0(v)$. This can be done since the cardinality of $V_{\infty}^2$ is not greater than $\bigcup\{\mathcal{B}(v):v\in V_{\infty}^2\}$, for each component $B\in \mathcal{B}$ is adjacent to at most countably many vertices of $V_{\infty}$. We next replace the components in $\mathcal{B}_0(v)$ by finite unions $B$ of them such that the degree of $v$ in $G[B\cup\{v\}]$ is at least 4, thus replacing the set $\mathcal{B}_0(v)$ by a~set $\mathcal{B}'_0(v)$.

We now argue as follows. For every vertex $v\in V_{\infty}^2$ and every subgraph $B\in \mathcal{B}'_0(v)$, consider a subgraph $B'$ of $G$ induced by $V(B)$ and all vertices having a neighbor in $B$. Hence, $d_{B'}(v)\geq 4$. The subgraph $B'$ admits a 3-edge coloring such that every its vertex with degree at least 4 is majority colored. This way, all vertices of $B$ are majority colored. This follows from Lemma~\ref{f3} if $B'$ is finite, or from Lemma~\ref{count3} if $B'$ is countably infinite. Moreover, every vertex $v\in V_{\infty}^2$ is majority colored regardless of colors of its other incident edges because for any two colors, out of three, the cardinality of the set of  edges incident to $v$ with these two colors equals $d_G(v)$.

Suppose that there remain components $B\in \mathcal{B}$ whose edges are not colored yet. For every such component, we analogously consider a subgraph $B'$ of $G$ induced by $V(B)$ and all vertices having a neighbor in $B$. By Lemma~\ref{f3} or Lemma~\ref{count3}, there exists a~3-edge-coloring of $B'$ such that every vertex with degree at least 4 in $B'$ is majority colored. Therefore, every vertex of $B$ is majority colored. This is also the case for every vertex $v\in V(B)\cap(V_{\infty}\setminus (V_{\infty}^1\cup V_{\infty}^2))$ since the degree of such a vertex in $B'$ is at least 4. 

Consequently, every vertex of $V_{\infty}$ is now majority colored in $G$ because, for every vertex of $V_{\infty}\setminus (V_{\infty}^1\cup V_{\infty}^2)$, the set of its incident edges is partitioned into subsets in which the cardinality of edges of the same color is not greater than the cardinality of edges with other colors. Thus we obtain a majority 3-edge-coloring of $G$.
\end{proof}

We conclude with the following conjecture for finite and infinite graphs.
\begin{conj}
Every graph with minimum degree at least $4$ admits a majority edge-coloring from lists of size $3$.
\end{conj}

\bibliographystyle{abbrv}
\bibliography{lit.bib}

\end{document}